\documentclass[12pt, reqno]{amsart}
\usepackage{amsmath, amsthm, amscd, amsfonts, amssymb, graphicx, color}
\usepackage[ plainpages]{hyperref}

\newtheorem{theorem}{Theorem}[section]

\newtheorem{proposition}[theorem]{Proposition}

\theoremstyle{definition}

\theoremstyle{remark}

\usepackage[all]{xy}
\numberwithin{equation}{section}

\begin{document}
\setcounter{page}{1}

\title[ On The Farthest point Problem in Banach spaces ]
{ On The Farthest point Problem in Banach spaces }

\author[ A. Yousef, R. Khalil and B. Mutabagani ]{  A. Yousef$^{1}$, R. Khalil$^{2}$and B. Mutabagani$^{3}$}

\address{$^{1}$ Department of Mathematics, The University of Jordan , Al Jubaiha, Amman 11942, Jordan.} \email{\textcolor[rgb]{0.00,0.00,0.84}{abd.yousef@ju.edu.jo}}

\address{$^{2}$ Department of Mathematics, The University of Jordan , Al Jubaiha, Amman 11942, Jordan.} \email{\textcolor[rgb]{0.00,0.00,0.84}{roshdi@ju.edu.jo}}

\address{$^{3}$ Department of Mathematics, The University of Jordan , Al Jubaiha, Amman 11942, Jordan.} \email{\textcolor[rgb]{0.00,0.00,0.84}{almansor.326@gmail.com}}

\begin{abstract}
If every point in a normed space X admits a unique farthest point in a given bounded subset E, then must
E be a singleton?. This is known as the farthest point problem. In an attempt to solve this problem, we give our contribution toward solving it, in the positive direction, by proving that every such subset $E$ in the sequence space $\ell^1$ is a singleton.
\end{abstract}\maketitle

\maketitle
\section{Introduction}

Let $X$ be a normed space, and $E$ be a closed and bounded subset of $X$. We define the real valued function $D(.,E): X\to \mathbb{R}$ by $$D(x,E)=\sup\{\|x-e\|: e\in E\},$$ the farthest distance function. We say that $E$ is remotal if for every $x\in X$, there exists $e\in E$ such that $D(x,E)=\|x-e\|$. In this case, we denote the set $\{e\in E: D(x,E)=\|x-e\|\}$ by $F(x,E)$. It is clear that $F(.,E):X\to E$ is a multi-valued function. However, if $F(.,E): X\to E$ is a single-valued function, then $E$ is called uniquely remotal. In such case, we denote $F(x,E)$ by $F(x),$ if no confusion arises.\\

The study of remotal and uniquely remotal sets has attracted many mathematicians in the last decades, due to its connection with the geometry of Banach spaces.
We refer the reader to \cite{aliansari},  \cite{baronti},  \cite{rao}, \cite{narang} and \cite{sk} for samples of these studies. However, uniquely remotal sets
are of special interest. In fact, one of the most interesting and hitherto unsolved problems in the theory of farthest points, known as the the farthest point problem, which is stated as:  If every point of a normed space $X$ admits a unique farthest point in a given bounded subset $E$, then must $E$ be a singleton ?.\\

This problem gained its importance when Klee \cite{klee} proved that:  singletoness of uniquely remotal sets is equivalent to convexity of Chybechev sets in Hilbert spaces (which is an open problem too, in the theory of nearest points).\\

Since then, a considerable work has been done to answer this question, and many partial results have been obtained toward solving this problem. We refer the reader
 to \cite{aliansari}, \cite{baronti}, \cite{narang} and \cite{sk} for some related work on uniquely remotal sets.\\

Centers of sets have played a major role in the study of uniquely remotal sets, see \cite{aliansari}, \cite{aliansari2} and \cite{baronti}. Recall that a center
 $c$ of a subset $E$ of a normed space $X$ is an element $c\in X$ such that  $$D(c,E)=\inf_{x\in X} D(x,E).$$ Whether a set has a center or not is another question. However, in inner
  product spaces, any closed bounded set does have a center \cite{aliansari}.\\

In \cite{niknam} it was proved that if $E$ is a uniquely remotal subset of a normed space, admitting a center $c$, and if $F$, restricted to the line segment
$[c,F(c)]$ is continuous at $c$, then $E$ is a singleton. Then recently, a generalization has been obtained in \cite{syk}, where the authors proved the singletoness of uniquely remotal sets if the farthest point mapping $F$ restricted to $[c,F(c)]$ is partially continuous at $c$. Furthermore, a generalization of Klee's result in \cite{klee}, "If a compact subset $E$, with a center $c$, is uniquely remotal in a normed space $X$, then $E$ must be a
singleton", was also obtained in \cite{syk}.\\

In this article, we prove that every uniquely remotal subset of the sequence space $\ell^1(\mathbb{R})$ is a singleton. Recall that $\ell^1(\mathbb{R})=\{ x=(x_n):~x_n\in \mathbb{R}~ and ~\sum_{n=1}^\infty |x_n|<\infty\}$.\\
\section{ Preliminaries}

In this section, we prove the following propositions that play a key role in the proof of the main result.

Throughout the rest of the paper, $F$ will denote the farthest distance single-valued function associated with a uniquely remotal set $E$.

\begin{proposition}
\label{prop1}
Let $E$ be a uniquely rematal subset of a Banach space $X$. Let $(x_n)$ be a sequence in $X$ such that $(x_n)$ converges to $x\in X$. If $F(x_n)=y$ for all $n$, where $y\in E$, then $F(x)=y$.

\end{proposition}
\begin{proof}
Suppose that $F(x)\neq y$. Since $E$ is uniquely remotal, then there exists $w\in E$ such that $F(x)=w$. Further, there exists $\epsilon>0$ such that $||x-w||>||x-y||+\epsilon.$ Also, there exists $n_0\in \mathbb{N}$ such that $||x_n-x||<\frac{\epsilon}{2}$ for all $n\geq n_0$. Therefore, for $m\geq n_0$

\begin{eqnarray*}
||x_m-w||&\geq& ||x-w||-||x_m-x||\\
&>& ||x-y||+\epsilon-\frac{\epsilon}{2}\\
&>& ||x_m-y||+\frac{\epsilon}{2}-||x_m-x||\\
&>& ||x_m-y||.
\end{eqnarray*}
 This contradicts that $y=F(x_m)$. Hence, we must have $F(x)=y$.\\
\end{proof}
\begin{proposition}
\label{prop2}
Let $K$ be a compact subset of a Banach space $X$ and $E$ be uniquely remotal in $X$. Then there exist $x\in K$ and $e\in E$ such that $$D(E,K)=\sup\{||y-\theta||:y\in K,~\theta\in E\}=||e-x||.$$

\end{proposition}

\begin{proof}
From the definition of $D(E,K)$, there exist two sequences $(e_n)$ and $(x_n)$ in $E$ and $K$ respectively such that $$D(E,K)=\lim_{n\rightarrow \infty} ||e_n-x_n||.$$
Since $K$ is compact, then there exists a subsequence $(x_{n_k})$ of $(x_n)$ such that $(x_{n_k})$ converges to $x$ in $K$. So, $$D(E,K)=\lim_{k\rightarrow \infty} ||e_{n_k}-x_{n_k}||.$$

The definition of $D(E,K)$ implies that $D(E,K)\geq ||e'-x'||$ for all $e'\in E$ and $x'\in K$. Therefore, $$\displaystyle \lim_{k\rightarrow \infty}||e_{n_k}-x_{n_k}||\geq ||x-F(x)||.$$ But $$||e_{n_k}-x_{n_k}||\leq ||e_{n_k}-x||+||x-x_{n_k}||\leq ||x_{n_k}-x||+||x-F(x)||.$$
Thus
$$\lim_{k\rightarrow \infty}||x_{n_k}-y_{n_k}||\leq ||x-F(x)||.$$

Since $x\in K$ and $F(x)\in E$, it follows that $D(E,K)=||x-F(x)||$, which ends the proof.\\
\end{proof}

\section{Main Results}

Let $E$ be a uniquely remotal subset of a Banach space $X$. Let $x_0$ be an element in $X$ and $e_0\in E$ be the unique farthest point from $x_0$, i.e $F(x_0)=e_0$. Consider the closed ball $$B[x_0,||x_0-e_0||]=B[x_0,D(x_0,E)].$$ Then clearly $e_0$ lies on the boundary of $B[x_0,D(x_0,E)]$.\\

Let $J=\{B[y,||y-e_0||:F(y)=e_0\},$ and define the relation $"\leq"$ on $J$ as follows: $$B_1\leq B_2~~if~~B_2\subseteq B_1.$$ It is easy to see that the relation $"\leq"$ is a partial order.\\

Now, we claim the following.
\begin{theorem}
\label{thm1}
J has a maximal element.
\end{theorem}
\begin{proof}
Let $T$ be any chain in $J$. Consider the net $\{||y_\alpha-e_0||:\alpha \in I\}$. Notice that if $B_{\alpha_1}\leq B_{\alpha_2}$ then $||y_{\alpha_2}-e_0||\leq ||y_{\alpha_1}-e_0||$. Let $\displaystyle r=\inf_{\alpha \in I}||y_\alpha-e_0||$. Then it is easy to see that if the infimum is attained at some $\alpha_0$, then $B_{\alpha_0}[y_{\alpha_0},||y_{\alpha_0}-e_0||]$ is an upper bound for $T$. If the infimum is not attained then there exists a sequence $(B_n)$ in $T$ such that $\displaystyle \lim_{n\rightarrow \infty}||y_n-e_0||=\inf_{\alpha\in I}||y_\alpha-e_0||=r$.\\

We claim that $(y_n)$ has a convergent subsequence. If not, then there exists $\epsilon>0$ such that $||y_n-y_m||>\epsilon$ for all $n,~m$. Clearly we can assume that $\epsilon<r$. \\

Since $\displaystyle \lim_{n\rightarrow \infty}||y_n-e_0||=r$, then there exists $n_0\in \mathbb{N}$ such that $||y_n-e_0||<r+\frac{\epsilon}{2}$ for all $n\geq n_0$. But $||y_{n_0}-y_{{n_0}+1}||>\epsilon$, so $B_{n_0}\subseteq B_{{n_0}+1}$. Farther, $r\leq ||y_{n_0}-e_0||$ and $||y_{{n_0}+1}-e_0||<r+\frac{\epsilon}{2}.$  Without loss of generality, we can assume, for simplicity, that $y_{n_0}=0$. Then the element $v=(1+\frac{r}{||y_{{n_0}+1}||}y_{{n_0}+1})\in B_{{n_0}+1}$. \\

Now, $||v-0||=||v||=||y_{{n_0}+1}||+r>r+\epsilon$. Thus, $v\not\in B_{n_0}$ which contradicts the fact that $B_{n_0+1}\subseteq B_{n_0}$. Hence, there is a subsequence $(y_{n_k})$ that converges to some element, say $y$. By assumption $F(y_{n_k})=e_0$ for all $n_k$, which implies by Proposition \ref{prop1} that $F(y)=e_0$. Thus, $B[y,||y-e_0||]\in J$. \\

It suffices now to show that $B[y,||y-e_0||]\subseteq B_\alpha$ for all $\alpha\in I$. If this is not true then there exists $w\in B[y,||y-e_0||]$ such that $w\not\in B_{m_1}$ for some $m_1$. Since $(B_n)$ is a chain, then $w\not\in B_{n_k}$ for all $n_k>m_1$. Furthermore, $||w-y_{n_k}||>r+\epsilon'$ for some $\epsilon'>0$ and all $n_k>m_1$. \\

But $||w-y_{n_k}||\leq ||y_{n_k}-y||+||y-w||$, where $||y_{n_k}-y||\rightarrow 0$ and $||y-w||<||y-F(y)||=||y-r||$. It follows that $\displaystyle \liminf_{n_k}||w-y_{n_k}||\leq r$, which contradicts the fact that $||w-y_{n_k}||>r+\epsilon'$. This means that $B[y,||y-e_0||]$ is an upper bound for the chain $T$. Hence, By Zorn's lemma $J$ has a maximal element.\\

\end{proof}

Now we are ready to prove the main result of this paper.
\begin{theorem}
Every uniquely remotal set in $\ell^1(\mathbb{R})$ is a singleton.
\end{theorem}
\begin{proof}
Let $E$ be a uniquely remotal set in $\ell^1$, and let $\hat{e}$ be the unique farthest point in $E$ from $0$, i.e. $F(0)=\hat{e}$. By Theorem \ref{thm1}, $J=\{B[y,||y-\hat{e}||]:F(y)=\hat{e}\}$ has a maximal element say $B[\hat{v},||\hat{v}-\hat{e}||]$. \\

Without loss of generality, we may assume that $\hat{v}=0$ and $||\hat{e}||=1$ so that the maximal element is the unit ball of $\ell^1$. Let $\hat{e}=(b_1,b_2,b_3,\dots)$. Since $||\hat{e}||=1$ then with no loss of generality we can assume that $b_1\neq 0$. Further, assume $b_1>0$. So, $b_1>\frac{1}{m_0}$ for some $m_0\in \mathbb{N}$. \\

 Let $\delta_1=(1,0,0,\dots )$ and consider the sequence $(\frac{\delta_1}{n})$ in $\ell^1$, where $n>m_0$. Then $F(\frac{\delta_1}{n})\neq \hat{e}$ for all $n>m_0$, since if $F(\frac{\delta_1}{n})=\hat{e}$ for some $n>m_0$, then for $w\in B[\frac{\delta_1}{n},||\frac{\delta_1}{n}-\hat{e}||]$, we have $||w||-||\frac{\delta_1}{n}||\leq ||w-\frac{\delta_1}{n}||\leq||\frac{\delta_1}{n}-\hat{e}||$. But $b_1>\frac{1}{n}$, so $||\frac{\delta_1}{n}-\hat{e}||=||\hat{e}||-\frac{1}{n}=||\hat{e}||-||\frac{\delta_1}{n}||$. Thus, $||w||\leq ||\hat{e}||=1$ and accordingly $w\in B[0,1]$, which contradicts the maximality of $B[0,1]$. Hence, $F(\frac{\delta_1}{n})\neq \hat{e}$ for all $n>m_0$.\\

 Let $F(\frac{\delta_1}{n})=z_n=(c_1^n,c_2^n,c_3^n,\dots)$. Then we must have $c_1^n<\frac{1}{n}$ for all $n>m_0$. Otherwise, we obtain that $||z_n-\frac{\delta_1}{n}||=||z_n||-||\frac{\delta_1}{n}||\leq 1-\frac{1}{n}=||\hat{e}-\frac{\delta_1}{n}||$, which contradicts the fact that $F(\frac{\delta_1}{n})=z_n$.

  Now, since $\frac{\delta_1}{n}\rightarrow 0$, then $||z_n||\rightarrow 1$. Further, the sequence $c_1^{n}$ converges to $ \lambda $, where $\lambda \leq 0$. \\

 Consider the set $P=\{b_1\delta_1\}$. Then, clearly $D(\hat{e},P)=\sum_{j=2}^\infty |b_j|<1$. Also, $D(z_n,P)=||z_n-b_1\delta_1||=|c_1^n-b_1|+\sum_{j=2}^\infty |c_j^n|$. Therefore,

  \begin{eqnarray*}
\lim_{n\rightarrow\infty} D(z_n,P)&=& (b_1+|\lambda|)+\lim_{n\rightarrow \infty} \sum_{j=2}^\infty |c_j^n|\\
  &=& b_1+|\lambda|+(1-|\lambda|)\\
 &=& 1+b_1
\end{eqnarray*}

 Since $D(P,E)\geq D(P,z_n)$ for all $n$, we get that $D(P,E)\geq1+b_1$. On the other hand, $\displaystyle D(P,E)=\sup_{e\in E}||b_1\delta_1-e||\leq b_1+1$, since $||e||\leq 1$ for every $e\in E$. Thus,
$$
 D(P,E)=1+b_1.
 $$
 By Proposition \ref{prop2}, $D(P,E)=||b_1\delta_1-e'||$ for some $e'\in E$. So, $$1+b_1\leq b_1+||e'||\leq 1+b_1,$$ which implies that $||e'||=1$. Therefore, $e'$ is another farthest point in E from $0$, i.e. $F(0)=\{e', \hat{e}\}$, which contradicts the unique remotality of $E$. Hence, $E$ must be a singleton.

\end{proof}
{\small

\end{document}